\documentclass[reqno,twoside,12pt]{amsart}
\hoffset - 1.5cm

\usepackage{amsmath}
\usepackage{amssymb}
\usepackage{graphicx}
\usepackage{amstext}
\usepackage{amsopn}
\usepackage{amsfonts}
\usepackage{xcolor}
\usepackage{setspace}
\usepackage{enumerate}
\usepackage[polish,english]{babel}
\usepackage[utf8]{inputenc}
\usepackage{polski}

\newtheorem{Remark}{Remark}[section]
\newtheorem{Corollary}[Remark]{Corollary}
\newtheorem{Definition}[Remark]{Definition}

\newtheorem{Lemma}[Remark]{Lemma}

\newtheorem{Theorem}[Remark]{Theorem}

\newcommand{\bE}{\mathbb{E}}

\newcommand{\bH}{\mathbb{H}}

\newcommand{\bN}{\mathbb{N}}

\newcommand{\bQ}{\mathbb{Q}}
\newcommand{\bR}{\mathbb{R}}

\newcommand{\bW}{\mathbb{W}}
\newcommand{\bV}{\mathbb{V}}

\newcommand{\bZ}{\mathbb{Z}}

\newcommand{\cC}{\mathcal{C}}

\newcommand{\cH}{\mathcal{H}}

\newcommand{\cU}{\mathcal{U}}
\newcommand{\cV}{\mathcal{V}}
\newcommand{\cW}{\mathcal{W}}

\numberwithin{equation}{section} \errorcontextlines=0

\newcommand{\nt}{\noindent}

\newcommand{\sub}{\overline{\operatorname{sub}}}

\begin{document}

\title[Connected sets of solutions]{Connected sets of solutions of symmetric elliptic systems}

\author{Anna Go\l\c{e}biewska}
\address{Faculty of Mathematics and Computer Science\\
Nicolaus Copernicus University \\
PL-87-100 Toru\'{n} \\ ul. Chopina $12 \slash 18$ \\
Poland, ORCID 0000--0002--2417--9960}

\email{Anna.Golebiewska@mat.umk.pl}

\author{S{\l}awomir Rybicki}

\address{Faculty of Mathematics and Computer Science\\
Nicolaus Copernicus University \\
PL-87-100 Toru\'{n} \\ ul. Chopina $12 \slash 18$ \\
Poland, ORCID 0000--0003--3520--0515}

\email{rybicki@mat.umk.pl}

\author{Piotr Stefaniak}
\address{School of Mathematics,
West Pomeranian University of Technology \\PL-70-310 Szcze\-cin, al. Piast\'{o}w $48\slash 49$, Poland, ORCID 0000--0002--6117--2573}

\email{pstefaniak@zut.edu.pl}

\numberwithin{equation}{section}
\allowdisplaybreaks
\date{\today}

\maketitle

\begin{abstract}
The purpose of this paper is twofold. First we study bifurcations of connected sets of critical orbits of some invariant functional from a given family of critical orbits. We use techniques of equivariant bifurcation theory to obtain a Rabinowitz type alternative for symmetric gradient operators. The second aim is to apply the abstract results to studying orbits of nonconstant solutions of a nonlinear Neumann problem.
\end{abstract}

\section{Introduction}

The aim of this paper is to study the global bifurcation of solutions of a parameter dependent equation  in the presence of symmetry of some compact Lie group. More precisely, assuming the existence of the so called set of trivial solutions, we prove a necessary and sufficient conditions for  bifurcation of solutions  from this set and describe the global behaviour of a bifurcating continuum (i.e. a closed, connected set).

The most famous result concerning the global bifurcation is the classical Rabinowitz theorem, see \cite{Rabinowitz1}. Rabinowitz has considered the equation in general form 
\begin{equation}\label{eq:Rab}
T(u, \lambda)=0,
\end{equation}
where $T(u, \lambda)=u-\eta(u, \lambda),$ for $(u,\lambda) \in \bE \times \bR$, $\bE$ is a real Banach space and $\eta$ is a compact and continuous operator of the form $\eta(u, \lambda)=\lambda Lu+o(\|u\|)$ with $L$ being a compact linear operator. The set of trivial solutions has been assumed to be $\{0\} \times \bR.$ In this setting it has been proved that if $\mu$ is a characteristic value of $L$ of odd algebraic multiplicity, then there exists a continuum of solutions of problem \eqref{eq:Rab} bifurcating from  $(0,\mu)$. Moreover, this continuum is either unbounded or it meets $(0,\hat{\mu})$, where $\hat{\mu}\neq \mu$ is another characteristic value.

The above result has been generalised by many authors, considering different classes of operators appearing in the equation \eqref{eq:Rab}. In particular, a theorem of this type has been proved  in \cite{Ryb2005milano} for operators $T$ as in \eqref{eq:Rab}, but additionally gradient and with symmetries of some compact Lie group $G$.  Namely, it has been stated that if $\mu$ is a characteristic value of odd multiplicity or the eigenspace associated with this value is a nontrivial representation of $G$, then from $(0,\mu)$ bifurcates a continuum which is either unbounded or meets another characteristic value. 
In \cite{GolRyb2011} an analogous result has been given for an operator $T$ being a gradient of a strongly indefinite invariant functional.

The general form of the equation \eqref{eq:Rab} appears in the natural way when considering a wide class of differential equations, for example the elliptic systems with Dirichlet boundary conditions. Therefore, the abstract results described above could be applied to studying sets of solutions of such systems, see \cite{GolRyb2011}, \cite{Ryb2005milano}. Moreover, with an easy assumption on the system, we can obtain that $\{0\} \times \bR$ is a set of trivial solutions of some associated operator equation of the form \eqref{eq:Rab} and that there exist levels $\lambda \in \bR$ such that solutions of the form $(0, \lambda)$ are isolated in $\bE \times \{\lambda\}$.  This allows to apply classical methods, like the mentioned Rabinowitz theorem, to study bifurcations.

However, if we consider an elliptic  Neumann problem of the form
\begin{equation*}
\left\{
\begin{array}{rclcl}   -\triangle u & =& \nabla_u F(u,\lambda )   & \text{ in   } & \cU \\
                   \frac{\partial u}{\partial \nu} & =  & 0 & \text{ on    } & \partial\cU,
\end{array}\right.
\end{equation*} where the potential $F$ is assumed to be $G$-invariant, for $G$ being a compact Lie group, and some additional assumptions on $\cU$ and $F$ are satisfied (see Section \ref{sec:elsys}),
 we obtain the set of trivial solutions in a different form. More precisely, the trivial solutions form orbits of the action of the group, so we study the bifurcation from the set $G(u_0) \times \bR$ for some given orbit $G(u_0)$ of constant solutions. In this situation it can happen that the orbit $G(u_0)$ is a manifold of positive dimension. For instance if $G=SO(3)$ and the isotropy group $G_{u_0}$ of $u_0$  equals $SO(2)$, then the orbit $G(u_0)$ is $G$-homeomorphic to $G \slash G_{u_0} = S^2.$ Therefore the classical results mentioned above can not be applied in this situation.
 
That is why the theory of bifurcation from the orbit has been recently developed, see \cite{PRS}, \cite{GolKluSte}, \cite{GolSte}. In particular, in \cite{GolKluSte} we have studied global bifurcations from an orbit of solutions of a nonlinear Neumann problem defined on a ball, proving some version of a Rabinowitz type theorem. A similar result for a non-cooperative Neumann problem has been given in \cite{GolSte}. These results have been proved with the assumption that the isotropy group $G_{u_0}$ is trivial.

In the current paper we study the global bifurcation problem from an orbit $G(u_0)$ of any isotropy type. More precisely, we consider a gradient operator of the form of a completely continuous perturbation of the identity and consider its set of zeros of the form $G(u_0)\times\bR$, for $G$ being a compact Lie group. In Theorem \ref{thm:RabAlt} we formulate a sufficient condition for existence of a global branch of orbits of zeros bifurcating from $G(u_0)\times\bR$.

In Theorem \ref{thm:globbifV} we give a sufficient condition for the bifurcation in terms of eigenvalues and eigenspaces of the hessian of the functional. More precisely, we prove that if the space of fixed points of the action of the group $H$ on some direct sum of eigenspaces of the hessian is odd-dimensional, then the global bifurcation occurs, where $(H)$ is the orbit type of $G(u_0)$.
To prove this theorem we apply Theorem \ref{thm:main}, where we give a sufficient condition to distinguish the Euler characteristics of some $G$-CW-complexes.

Our second aim is to apply the abstract results to studying a nonlinear Neumann problem defined on an open set. In Theorem \ref{thm:globellip} we give sufficient conditions for global bifurcations of connected sets of orbits of nonconstant solutions, emanating from a critical orbit of constant solutions. We formulate these conditions in terms of eigenvalues of the hessian of the potential of the system as well as eigenvalues and eigenspaces of the Laplacian. We emphasise that in this theorem we do not assume anything about the orbit type of the orbit of constant solutions.

\section{Notation}

We start with recalling some notation concerning the theory of transformation groups.
Denote by $G$ a compact Lie group and by $\sub(G)$ the set of closed subgroups of $G$. We call subgroups $H,K\in \sub(G)$ conjugate if $H=gKg^{-1}$ for some $g\in G$. Denote by $(H)$ the conjugacy class of $H$ and by $\sub[G]$ the set of conjugacy classes of closed subgroups of $G$.

By a $G$-space we understand a Hausdorff space $X$ with an action of the group $G$. Since we consider Lie groups, we assume that $X$ is a manifold and the action is smooth. For a $G$-space $X$ and $x\in X$, we denote by $G_{x}$ the isotropy group of $x$, i.e. the set $\{g\in G\colon gx=x\}$ and by $G(x)$ the orbit $\{gx\colon g\in G\}$. Moreover, for a fixed $H\in\sub(G)$, by $X^H$ we denote the set of fixed points of the action of $H$, i.e. the set $\{x\in X\colon \forall_{h\in H}\, hx=x \}$ and we put $X_{(H)}=\{x\in X\colon (G_x)=(H)\}$.

Let $X$ be an $H$-space and consider the  product $G\times X$ with the $H$-action given by $(h,(g,x))\mapsto(gh^{-1}, hx)$ for $h\in H$, $(g,x)\in G\times X$. Denote by $G\times_H X$ the space of orbits of this action and note that $G\times_H X$ is a $G$-space with the $G$-action given by $(g',[g,x])\mapsto [g'g,x]$ for $g'\in G$, $[g,x]\in G\times_H X$. This space is called the induced $G$-space.

A similar construction can be done for pointed spaces. Recall that if $X$ is a pointed $H$-space, then the action on the base point of $X$ is assumed to be trivial.
If $Y$ is an $H$-space without a base point, then by $Y^+$ we denote a pointed space $Y\cup\{\ast\}$.
Suppose that $X$ is a pointed $H$-space and consider the smash product $G^+\wedge X$ with the $H$-action $(h,(g,x)) \mapsto(gh^{-1}, hx)$ for $h\in H$, $(g,x)\in G^+\wedge X$. We call the space of orbits of this action the smash product over $H$ and  denote it by $G^+\wedge_H X$. Note that $G^+\wedge_H X$ is a $G$-space with the $G$-action given by $(g',[g,x])\mapsto [g'g,x]$ for $g'\in G$, $[g,x]\in G^+\wedge_H X$.

The properties of induced $G$-spaces can be found for example in \cite{Kawakubo}, \cite{TomDieck}. For the properties of the smash product over $H$ we refer to \cite{TomDieck}.

Throughout this paper we denote by $B^k$ ($D^k$, $S^{k-1}$, respectively) the $k$-dimensional open unit ball (the $k$-dimensional closed unit ball, the $k-1$-dimensional unit sphere).

\section{The degree and the Euler characteristic}

The aim of our paper is to consider the global bifurcation phenomenon. As the main tool in this problem we use the degree theory for equivariant maps. More precisely, we work with  the degree for $G$-equvariant gradient maps defined  by Gęba in \cite{Geba} and its infinite dimensional generalisation for completely continuous perturbations of the identity, defined in \cite{Ryb2005milano}. These degrees are elements of the Euler ring $U(G)$. For the convenience of the reader we present below the definition and basic properties of this ring, as well as of the Euler characteristic $\chi_G(X)$ of a pointed $G$-CW-complex $X$. We refer to \cite{TomDieck1} and \cite{TomDieck} for more details.

\subsection{The Euler ring}

Fix a compact Lie group $G$.
 It is known that if $H\in\sub(G)$ is the isotropy group of $x$, then all the points of the orbit $G(x)$ have the isotropy groups conjugate to $H$. Therefore the class $(H)$ is sometimes called the orbit type of $G(x)$, or briefly the type of $G(x)$. In a similar way, in the case of a $G$-space $X$ with all the orbits having the same type, we call this conjugacy class the type of $X$. For example, if we consider the space $G/H$, with the action of $G$ given by $(g, g'H)\mapsto gg'H$, the type of this space is $(H)$. We can also consider the space $G/H \times B^k$, with the trivial $G$-action on $B^k$ and again obtain the space of the type $(H)$. This model space is used as the so called $k$-cell of the type $(H)$ in the definition of a $G$-CW-complex. We start with the definition of attaching a family of $k$-cells. Denote by $\sqcup$ the disjoint union.

\begin{Definition}
Let $(X,A)$ be a pair of $G$-spaces and $H_1,\ldots,H_q \in \sub(G)$. We say that $X$ is obtained from $A$ by attaching a family of $k$-cells of the type $\{(H_j)\colon j=1,\ldots,q\}$, if there exists a $G$-equivariant map
\[\varphi \colon \left(\bigsqcup\limits^q_{j=1}D^k\times G/H_j,\bigsqcup\limits^q_{j=1}S^{k-1}\times G/H_j\right) \to (X,A)\]
which maps homeomorphically $\bigsqcup\limits^q_{j=1}B^k\times G/H_j$ onto $X \setminus A$.
\end{Definition}

\begin{Definition}
Let $X$ be a $G$-space with a base point $*\in X$. If there exists a finite sequence of $G$-spaces
$X_0 \subset X_1 \subset \ldots \subset X_p = X$
such that
\begin{enumerate}
\item $X_0$ is $G$-homeomorphic with $\{*\} \sqcup \bigsqcup\limits^{q(0)}_{j=1} G/H_{j,0}$, where $H_{1,0}, \ldots , H_{q(0),0} \in \sub(G)$,
\item $X_k$ is obtained from $X_{k-1}$ by attaching a family of $k$-cells of the type $ \{(H_{j,k})\colon j=1, \ldots,q(k)\}$ for $k=1,\ldots, p$,
\end{enumerate}
then we call $X$ a pointed $G$-CW-complex.

The set $\bigcup\limits^p_{k=0}\bigcup\limits^{q(k)}_{j=1}\{(k,(H_{j,k}))\}$ is called the type of the cell decomposition of $X$. 
\end{Definition}

Denote by $[X]$ the $G$-homotopy class of a pointed $G$-CW-complex $X$. Let $\mathbf{F}$ be the free abelian group generated by the pointed $G$-homotopy types of $G$-CW-complexes and $\mathbf{N}$ the subgroup of $\mathbf{F}$ generated by all elements $[A]-[X]+[X/A]$ for pointed $G$-CW-subcomplexes $A$ of a pointed $G$-CW-complex $X$.

\begin{Definition}
Put $U(G)=\mathbf{F}/\mathbf{N}$ and let $\chi_G(X)$ be the class of $[X]$ in $U(G)$. The element $\chi_G(X)$ is said to be the $G$-equivariant Euler characteristic of a pointed $G$-CW-complex $X$.
\end{Definition}

For pointed $G$-CW-complexes $X,Y$ the actions in $U(G)$ are defined by
\begin{equation}\label{eq:actionsUG}
\left.\begin{array}{rcl}
\chi_G(X)+ \chi_G(Y)&=&\chi_G(X\vee Y),\\
\chi_G(X)\star \chi_G(Y)&=&\chi_G(X\wedge Y),
\end{array} \right.
\end{equation}
where $X\vee Y$ is the wedge sum and $X\wedge Y$ is the smash product.

\begin{Lemma}
$(U(G),+,\star)$, with the actions given by \eqref{eq:actionsUG}, is a commutative ring with unit $\mathbb{I}=\chi_G(G/G^+)$.
\end{Lemma}

We call $(U(G),+,\star)$ the Euler ring of $G$.

\begin{Lemma}
$(U(G),+)$ is a free abelian group with the basis $\chi_G(G/H^+),\ (H)\in\sub[G]$.
Moreover, if $X$ is a pointed $G$-CW-complex and $\bigcup\limits^p_{k=0}\bigcup\limits^{q(k)}_{j=1}\{(k,(H_{j,k}))\}$ is  the type of the cell decomposition of $X$, then 
\begin{equation*}
\chi_G\left(X\right)= \sum_{(H)\in\sub[G]}\nu_G\left(X,H\right)\chi_G\left(G/H^+\right),
\end{equation*}
where
$\nu_G(X,H)= \sum\limits^{p}_{k=0}(-1)^{k}\nu_G(X,H,k)$ and $\nu_G(X,H,k)$ is the number of $k$-dimensional cells of the type $(H)$.
\end{Lemma}

\subsection{The degree for equivariant gradient maps}

Let $\bV$ be a finite dimensional orthogonal representation of $G$ and $\Omega\subset\bV$ a $G$-invariant set.
Fix a $G$-invariant map $\varphi\in C^2(\Omega,\bR)$ (i.e. satisfying $\varphi(g x)=\varphi(x)$ for every $x\in\Omega$, $g\in G$) and note that  the gradient of a $G$-invariant map is $G$-equivariant, i.e.  $\nabla\varphi(g x)=g \nabla\varphi(x)$ for every $x\in\Omega$, $g\in G$.
Suppose that $(\nabla\varphi)^{-1}(0)\cap\partial\Omega=\emptyset$. Under these assumptions, there is defined the degree for $G$-equivariant gradient maps $\nabla_{G}\textrm{-}\mathrm{deg}(\nabla\varphi, \Omega) \in U(G)$, see \cite{Geba} for the definition and properties.

In general situation, computing the degree is a difficult problem. One of the methods of obtaining some information about the degree is to use its relation with the equivariant Conley index (see for example \cite{Geba} for the definition of the index), given by the Euler characteristic. Below we recall this relation.

Assume that $v_0$ is an isolated, non-degenerate critical point of $\varphi$. Denote by $\cW$ the direct sum of the eigenspaces of $\nabla^2\varphi(v_0)$ corresponding to the negative eigenvalues and let $S^{\cW}=\cW \cup \{\infty\}$ be a one-point compactification of $\cW$. Then $S^{\cW}$ is a pointed $G$-CW-complex, with $\infty$ as a base point. It is known that the equivariant Conley index of $\{v_0\}$ of the flow generated by $-\nabla \varphi$ is the $G$-homotopy type of $S^{\cW}$, see for example Theorem 5.2 of \cite{GolKluSte}. Moreover, if $\Omega \subset \bV$ is such that $(\nabla \varphi)^{-1}(0) \cap cl(\Omega)=\{v_0\}$, then (see \cite{Geba} and Corollary 1 of \cite{GolRyb2013})
\begin{equation}\label{eq:degchar}
\nabla_{G}\textrm{-}\mathrm{deg}(\nabla\varphi, \Omega)=\chi_G(S^{\cW}).
\end{equation}

We are going to generalise the formula \eqref{eq:degchar} to the case of a critical orbit.
Assume that $G(v_0)\subset(\nabla\varphi)^{-1}(0)$ is a non-degenerate critical orbit, i.e. $\dim \ker \nabla^2 \varphi(v_0)=\dim G(v_0)$, and assume that $\Omega\subset\bV$ is such that $(\nabla\varphi)^{-1}(0)\cap  cl(\Omega)= G(v_0)$.
Denote by $T_{v_0} G(v_0)$ the tangent space to the orbit $G(v_0)$ at $v_0$ and let $\bW=(T_{v_0} G(v_0))^{\bot}$. Put $H=G_{v_0}$ and note that $\bW$ is an orthogonal $H$-representation.
As before, we denote by $\cW$ the direct sum of the eigenspaces of $\nabla^2\varphi(v_0)$ corresponding to the negative eigenvalues.

\begin{Theorem}\label{thm:degfiniterelation}
Under the above assumptions,
\begin{equation*}
\nabla_{G}\textrm{-}\mathrm{deg}\left(\nabla\varphi, \Omega\right)=\chi_G\left(G^+ \wedge_H S^{\cW}\right).
\end{equation*}
\end{Theorem}

\begin{proof}
Define $\psi\colon\bW\cap\Omega\to\bR$ by $\psi=\varphi_{|\bW\cap\Omega}$. It is easy to see that the functional $\psi$ is $H$-invariant and $v_0$ is an isolated critical point of $\psi$.
Moreover, from the definition of $\psi$ and since $\nabla^2\varphi_{|T_{v_0} G(v_0)}(v_0)=0$, the direct sum of the eigenspaces of $\nabla^2\psi(v_0)$ corresponding to the negative eigenvalues is equal to $\cW$. From the properties of flows induced by gradient operators, we conclude that $G(v_0)$ is an isolated invariant set in the sense of the equivariant Conley index theory for the flow induced by $-\nabla\varphi$.
Hence the Conley index of $\{v_0\}$ of the flow generated by $-\nabla\psi$
is the $H$-homotopy type of  $S^{\cW}$. 
Therefore, from Theorem 3.1 of \cite{PRS} we obtain that $G^+ \wedge_H S^{\cW}$ is the Conley index of $G(v_0)$ of the flow generated by $-\nabla\varphi$.
Applying the relation of the Conley index and the degree theory, we obtain the assertion.
\end{proof}

Our main goal is studying the global bifurcation problem. To apply the degree theory to this problem we need to distinguish degrees of different maps. From Theorem \ref{thm:degfiniterelation} we know that it can be brought to comparing the Euler characteristics of $G$-CW-complexes of the form $G^+ \wedge_H S^{\cW}$.

\begin{Theorem}\label{thm:main}
Fix orthogonal $H$-representations $\cV$, $\cW$. If $\dim \cV^H$ is odd, then $$\chi_G(G^+ \wedge_H S^{\cW\oplus\cV}) \neq \chi_G(G^+ \wedge _H S^{\cW}).$$
\end{Theorem}

\begin{proof}
The idea of the proof is to determine in the $G$-CW-complexes $G^+ \wedge_H S^{\cW\oplus\cV}$ and $G^+ \wedge _H S^{\cW}$ cells of the type $(H)$. We are going to show that in both complexes there is only one such cell. Knowing the dimensions of these cells, we will be able to compare the $(H)$-th coordinates of their Euler characteristics.

First, observe that $G^+ \wedge_H S^{\cV}=(G\times_H \cV)\cup\{\infty\}$. Indeed,
it is clear that, for a pointed $G$-CW-complex $X$, $G^+ \wedge _H X=G \times_H X/G \times_H \{*\}$ and therefore
\begin{align*}
G^+ \wedge_H S^{\cV}=G \times_H S^{\cV}/G \times_H \{\infty\}=G \times_H (\cV\cup\{\infty\})/G \times_H \{\infty\}= 
(G \times_H \cV)\cup \{\infty\}.
\end{align*}
Analogously $G^+ \wedge_H S^{\cW\oplus \cV}=(G\times_H (\cW\oplus\cV))\cup\{\infty\}$.

If $H=G$, then, from the above equalities, $G^+ \wedge_H S^{\cV}$ is $G$-homeomorphic to $S^{\cV}$ and  $G^+ \wedge_H S^{\cW\oplus \cV}$ to $S^{\cW\oplus \cV}$, see Lemma 1.87 of \cite{Kawakubo}. Therefore the assertion follows from Lemma 3.4 of \cite{GarRyb}.

Assume that $H\neq G$. Then, using the fact that $(G\times_H X)_{(H)}= G/H \times X^H$, see Lemma 4.17 of \cite{Kawakubo}, we obtain
\[(G^+ \wedge _H S^{\cW})_{(H)}=((G\times_H \cW)\cup\{\infty\})_{(H)} =G/H \times \cW^H\]
and
\[(G^+ \wedge_H S^{\cW\oplus \cV})_{(H)}=G/H \times (\cW\oplus\cV)^H=G/H \times (\cW^H\oplus\cV^H),\]
i.e. there is only one $(\dim\cW^H)$-cell of the type (H) in the $G$-CW-complex $G^+ \wedge _H S^{\cW}$ and there is only one $(\dim\cW^H+\dim\cV^H)$-cell of the type (H) in the $G$-CW-complex $G^+ \wedge_H S^{\cW\oplus \cV}$. Therefore
\[
\chi_G(G^+ \wedge_H S^{\cW})=(-1)^{\dim\cW^H}\chi_G(G/H^+)+\sum_{(K)\in\sub[G], (K) \neq (H)} \alpha_{(K)} \chi_G(G/K^+)
\]
and
\[
\chi_G(G^+ \wedge_H S^{\cW\oplus \cV})=(-1)^{\dim\cW^H+\dim\cV^H}\chi_G(G/H^+)+\sum_{(K)\in\sub[G], (K) \neq (H)}\beta_{(K)} \chi_G(G/K^+),
\]
where $\alpha_{(K)}, \beta_{(K)} \in \bZ$, which completes the proof.
\end{proof}

\begin{Corollary}
From the definition of the Euler characteristic and the above theorem it follows that if $\dim \cV^H$ is odd, then $G^+ \wedge_H S^{\cW\oplus\cV}$ and $G^+ \wedge _H S^{\cW}$ are not $G$-homotopy equivalent.
In particular, since $G^+ \wedge_H S^{\cW\oplus\cV}$ and $G^+ \wedge _H S^{\cW}$ can be considered as the Conley indices of some critical orbits, it allows to compare these indices.
\end{Corollary}

In the next section we are going to apply Theorem \ref{thm:main} to compare the degrees of some equivariant gradient maps and obtain as a consequence a global bifurcation theorem.

\subsection{Global bifurcations}\label{subsec:glob}

We start this subsection with recalling the definition and basic properties of the degree for $G$-equivariant gradient maps of the form of a completely continuous perturbation of the identity (see \cite{Ryb2005milano} for the details).
Let $\bH$ be an infinite dimensional, separable Hilbert space which is an orthogonal $G$-representation.
Assume that $\bH=cl(\bigoplus_{n=0}^{\infty} \bH_n),$ where all subspaces $\bH_n$ are disjoint finite dimensional $G$-representations and put $\bH^n = \bigoplus_{k=0}^{n} \bH_{k}$.
Denote by $\pi_n\colon\bH\to\bH$ the $G$-equivariant orthogonal projection such that $\pi_n(\bH)=\bH^n$.
We call $\{\pi_n\colon\bH \to \bH\colon n \in \bN \cup \{0\} \}$ a $G$-equivariant approximation scheme on $\bH$.

Assume that $\Omega\subset\bH$ is an open, bounded and $G$-invariant subset and $\xi\in C^1(\bH,\bR)$ is a $G$-invariant map such that $\nabla\xi\colon\bH\to \bH$ is a completely continuous, $G$-equivariant operator and  $(Id-\nabla\xi)^{-1}(0)\cap \partial\Omega=\emptyset$.
To define the degree of $Id-\nabla\xi$ on $\Omega$, consider the restrictions of this map to $\bH^n$. Since $\bH^n$ are finite dimensional $G$-representations, there is defined the Gęba's degree $\nabla_G\text{-}\deg(\pi_{n}(Id-\nabla\xi),\Omega\cap\bH^{n})$. It appears (see \cite{Ryb2005milano}) that these restrictions stabilise for $n$ sufficiently large, i.e.
there is $n_0$ such that
\[
\nabla_G\text{-}\deg(\pi_n(Id-\nabla\xi),\Omega\cap\bH^n)= \nabla_G\text{-}\deg(\pi_{n_0}(Id-\nabla\xi),\Omega\cap\bH^{n_0})
\]
for every $n\geq n_0$.
Therefore, we can define the degree of $Id-\nabla\xi$ on $\Omega$ by
\begin{equation*}
\nabla_G\text{-}\deg(Id-\nabla\xi,\Omega)= \nabla_G\text{-}\deg(\pi_{n_0}(Id-\nabla\xi),\Omega\cap\bH^{n_0}).
\end{equation*}

Using this degree we can study a global bifurcation problem. Below we give the definition of the global bifurcation from a critical orbit.

Consider a family of $G$-invariant functionals $\Phi\in C^2(\bH\times\bR,\bR)$ and suppose that there is $u_0\in\bH$ such that $\nabla_u \Phi(u_0,\lambda)=0$ for every $\lambda\in\bR$. The invariance of $\Phi$ implies that $G(u_0)\subset(\nabla_u \Phi(\cdot,\lambda))^{-1}(0)$ for every $\lambda\in\bR$. We call the elements of $G(u_0)\times\bR$ the trivial solutions of $\nabla_u \Phi(u,\lambda)=0$.

\begin{Definition}
We say that a global bifurcation of solutions of $\nabla_u \Phi(u,\lambda)=0$ occurs from the orbit $G(u_0)\times \{\lambda_0\}$, if there is a connected component $\cC(u_0,\lambda_0)$ of $cl \{(u, \lambda) \in (\bH \times \bR) \setminus (G(u_0) \times \bR)\colon \nabla_u\Phi(u, \lambda)=0\}$, containing $(u_0,\lambda_0)$, such that either $\cC(u_0,\lambda_0)\cap (G(u_0) \times (\bR\setminus \{\lambda_0\})) \neq \emptyset$ or $\cC(u_0,\lambda_0)$ is unbounded.
\end{Definition}

Note that if a global bifurcation occurs, we obtain a connected component for every $v\in G(u_0)$. In particular, if the group $G$ is connected then we obtain a connected set of nontrivial solutions bifurcating from $G(u_0)\times \{\lambda_0\}$.

Assume additionally that
$\nabla_u \Phi(u,\lambda)=u-\nabla_u \zeta (u,\lambda),$ where
 $\nabla_u\zeta\colon\bH\times\bR\to\bH$ is a completely continuous, $G$-equivariant operator.
Consider $\lambda_0\in\bR$ such that the orbit $G(u_0)$ is degenerate in $(\nabla_u \Phi(\cdot,\lambda_0))^{-1}(0)$. Suppose that there is $\varepsilon>0$ such that $G(u_0)$ is non-degenerate in $(\nabla_u \Phi(\cdot,\lambda))^{-1}(0)$ for every $\lambda\in[\lambda_0-\varepsilon,\lambda_0+\varepsilon]\setminus\{\lambda_0\}$.
Fix a $G$-invariant open set $\Omega\subset\bH$ such that $(\nabla_u\Phi(\cdot,\lambda_0\pm\varepsilon))^{-1}(0)\cap cl(\Omega)=G(u_0)$.
Below we formulate an equivariant version of the Rabinowitz global bifurcation theorem.

\begin{Theorem}\label{thm:RabAlt}
If $\nabla_G\text{-}\deg(\nabla_u\Phi(\cdot,\lambda_0-\varepsilon),\Omega)\neq \nabla_G\text{-}\deg(\nabla_u\Phi(\cdot,\lambda_0+\varepsilon),\Omega)$,
then a global bifurcation of solutions of $\nabla_u \Phi(u,\lambda)=0$ occurs from the orbit $G(u_0)\times \{\lambda_0\}$.
\end{Theorem}

The proof of this theorem is standard in the degree theory, see for instance \cite{Ize}, \cite{Nirenberg}, \cite{Rabinowitz}, \cite{Rabinowitz1}. In the case of the Leray-Schauder degree it can be found for example in \cite{Brown}.

Below we are going to formulate the global bifurcation theorem in terms of some eigenspaces of the operator $Id-\nabla^2_u \zeta.$

For $m \geq 1$ denote by $\cW_m(\lambda)$ a $G$-representation being the direct sum of the eigenspaces of $(Id-\nabla^2_u \zeta(\cdot, \lambda))_{|\bH^m}$ corresponding to the negative eigenvalues. Since $\nabla^2_u \zeta$ is a completely continuous operator, for every $\lambda$ there exists $m_0\in \bN$ such that $\cW_m(\lambda)=\cW_{m_0}(\lambda)$ for every $m \geq m_0$. Put $\cW(\lambda)=\cW_{m_0}(\lambda).$ Note that $\cW(\lambda)$ is the direct sum of the eigenspaces of $Id-\nabla^2_u \zeta(\cdot, \lambda)$ corresponding to the negative eigenvalues.

\begin{Theorem}\label{thm:globbifV}
Let the above assumptions on $\bH, \Phi, u_0, \lambda_0$ and $\varepsilon$ be satisfied.
Assume that there exists a $G$-representation $\cV(\lambda_0)$ such that $\cW(\lambda_0+\varepsilon)=\cW(\lambda_0 - \varepsilon) \oplus \cV(\lambda_0)$ or $\cW(\lambda_0-\varepsilon)=\cW(\lambda_0 + \varepsilon) \oplus \cV(\lambda_0)$ and put $H=G_{u_0}$. If $\dim \cV(\lambda_0)^H $ is odd, then a global bifurcation of solutions of $\nabla_u \Phi(u,\lambda)=0$ occurs from the orbit $G(u_0) \times \{\lambda_0\}.$
\end{Theorem}

\begin{proof}
From the definition  of the degree and Theorem \ref{thm:degfiniterelation} it follows that if $G(u_0)\times \{\lambda\}$ is a nondegenerate orbit, then
$$\nabla_{G}\textrm{-}\mathrm{deg}(Id- \nabla_u \zeta(\cdot, \lambda),\Omega)= \nabla_G\text{-}\deg(\pi_{n_0}(Id-\nabla_u\zeta(\cdot, \lambda)),\Omega\cap\bH^{n_0})=
\chi_G(G^+ \wedge_H S^{\cW_{n_0}(\lambda)}),$$
where $n_0$ is chosen as in the definition of the degree. It is easy to see that for such $n_0$ we have $\cW(\lambda)= \cW_{n_0}(\lambda)$. Therefore 
$$\nabla_{G}\textrm{-}\mathrm{deg}(Id- \nabla_u \zeta(\cdot, \lambda),\Omega)=\chi_G(G^+ \wedge_H S^{\cW(\lambda)}).$$
The rest of the proof is a consequence of Theorems \ref{thm:main} and \ref{thm:RabAlt}.
\end{proof}

\begin{Corollary}
Consider the case $\nabla_u \zeta(u, \lambda)=\lambda\nabla \xi(u),$ where
 $\nabla \xi$ is a completely continuous, $G$-equivariant operator. It is easy to observe that for $\lambda_0 >0$ we have $\cW(\lambda_0 +\varepsilon)=\cW(\lambda_0 - \varepsilon) \oplus \cV(\lambda_0),$ where $\cV(\lambda_0)$ is the eigenspace of $Id-\lambda \nabla^2 \xi$ corresponding to the zero eigenvalue. Analogously for $\lambda_0 <0$ we have $\cW(\lambda_0 -\varepsilon)=\cW(\lambda_0 + \varepsilon) \oplus \cV(\lambda_0).$  In both cases, from Theorem \ref{thm:globbifV} we obtain that if $\dim \cV(\lambda_0)^H$ is odd, then a global bifurcation of solutions of $\nabla_u \Phi(u,\lambda)=0$ occurs from the orbit $G(u_0) \times \{\lambda_0\}.$
Note that if the group is trivial, i.e. $G=\{e\}$, this result corresponds to a well known result of Rabinowitz, see \cite{Rabinowitz1}
\end{Corollary}

\begin{Remark}
The operator of the form $Id-\lambda \nabla \xi$ appears for example when considering elliptic systems with Dirichlet boundary conditions. For the system with Neumann boundary conditions the operator $\nabla_u \zeta$ is of the form $\nabla_u \zeta_1 - \lambda \nabla_u \zeta_2$, where operators $\nabla_u \zeta_1$ and $\nabla_u \zeta_2$ are completely continuous. Note that also in this case, if we denote by $\cV(\lambda_0)$ the eigenspace of $Id - \nabla^2_u \zeta(\cdot, \lambda_0)$ corresponding to the zero eigenvalue, from Theorem \ref{thm:globbifV} we obtain that if $\dim \cV(\lambda_0)^H$ is odd, then a global bifurcation of solutions of $\nabla_u \Phi(u,\lambda)=0$ occurs from the orbit $G(u_0) \times \{\lambda_0\}$, see Theorem \ref{thm:globellip}.
\end{Remark}

\section{Elliptic system}\label{sec:elsys}

Fix a compact Lie group $G$ and consider the system
\begin{equation}\label{eq:system}
\left\{
\begin{array}{rclcl}   -\triangle u & =& \nabla_u F(u,\lambda )   & \text{ in   } & \cU \\
                   \frac{\partial u}{\partial \nu} & =  & 0 & \text{ on    } & \partial\cU,
\end{array}\right.
\end{equation}
where $\cU\subset\bR^N$ is an open, bounded subset  with the piecewise $C^1$ boundary. We assume that
\begin{enumerate}
\item[(a1)] $F\in C^2(\bR^p\times\bR,\bR)$,
\item[(a2)] $u_0$ is a critical point of $F$ for all $\lambda \in \bR$ and there exists a symmetric matrix $A$ such that $\nabla_u^2 F(u_0,\lambda)=\lambda A$,
\item[(a3)] there exist $C>0$ and $1\leq s < (N+2)(N-2)^{-1}$ such that $|\nabla^2_u F(u,\lambda)|\leq C(1+|u|^{s-1})$ (if $N=2$, we assume that $s\in[1,+\infty))$,
\item[(a4)] $\bR^p$ is an orthogonal $G$-representation and $F$ is $G$-invariant with respect to the first variable, i.e.  $F(g u,\lambda) =F(u,\lambda)$ for every $g \in G$, $u\in\bR^p$, $\lambda\in\bR$.
\end{enumerate}
From (a2) and (a4) it follows that $\nabla_u F(g u_0,\lambda)=0$ for every $g\in G$, $\lambda\in\bR$, i.e. $G(u_0)\subset(\nabla_u F(\cdot,\lambda))^{-1}(0)$ for every $\lambda\in\bR$. We additionally assume:
\begin{enumerate}
\item[(a5)] the orbit $G(u_0)$ is non-degenerate for every $\lambda\in\bR$, i.e.  $\dim \ker \nabla^2 F (u_0,\lambda) = \dim G (u_0)$.
\end{enumerate}

Let $H^1(\cU)$ denote the standard Sobolev space with the inner product
\begin{equation*}
\langle \eta,\xi\rangle_{H^1(\cU)}=\int_{\cU}(\nabla \eta(x), \nabla \xi(x)) +\eta(x) \cdot \xi(x)\, dx.
\end{equation*}
Consider the space $\bH=\bigoplus_{i=1}^p H^1(\cU)$ with the inner product given by
\begin{equation*}
\langle u, v\rangle_{\bH} = \sum_{i=1}^p \langle u_i,v_i\rangle_{H^1(\cU)}.
\end{equation*}
Note that $\bH$ is an orthogonal $G$-representation, with the action $(g,u)(x)\mapsto g u(x)$ for $g\in G$, $u\in \bH$, $x\in \cU$.

It is known that weak solutions of the system \eqref{eq:system} are in one-to-one correspondence with critical points (with respect to $u$) of the functional $\Phi \colon \bH \times \bR \to \bR$ given by
\begin{equation}\label{eq:Phi}
\Phi(u,\lambda)=\frac{1}{2}\int_{\cU}\sum^p_{i=1}(|\nabla u_i(x)|^2)\, dx-\int_{\cU}F(u(x),\lambda)\, dx.
\end{equation}
From the assumption (a4) it follows that $\Phi$ is $G$-invariant.

\begin{Remark}
Note that the assumption (a2) implies that there exists $F_0\in C^2(\bR^p\times\bR,\bR)$ such that  $F(u,\lambda) = \frac{\lambda}{2} (Au, u)-\lambda (A u_0,u) + F_0(u-u_0,\lambda)$ and for every $\lambda\in\bR$ there hold $\nabla_u F_0(0,\lambda)=0$ and $\nabla^2_u F_0(0,\lambda)=0$.
\end{Remark}

Denote by $\tilde{u}_0\in\bH$ the constant function $\tilde{u}_0\equiv u_0$. Below we collect some standard properties of $\nabla_u \Phi$, see for example \cite{GolSte}.

\begin{Lemma}\label{lem:postacPhi}
Under the assumptions (a1)--(a3):
\begin{equation*}
\left.\begin{array}{ll}
\nabla_u\Phi(u,\lambda)=u-\tilde{u}_0-L_{\lambda A}(u-\tilde{u}_0)-\nabla_u\eta_0(u-\tilde{u}_0,\lambda),
\end{array}\right.
\end{equation*}
where
\begin{enumerate}
\item $\displaystyle{\langle L_{\lambda A} u, v \rangle_{\bH} = \int_{\cU} (u(x), v(x)) + (\lambda A u(x), v(x))\, dx}$ for all $v\in \bH$,
\item $L_{\lambda A}$ is a self-adjoint, bounded, completely continuous operator,
\item $\eta_0\colon \bH\times\bR\to \bR$ is defined by $\displaystyle{\eta_0(u,\lambda) =\int_{\cU} F_0(u(x),\lambda)\,dx}$,
\item $\nabla_u\eta_0\colon\bH\times\bR\to\bH$ is a completely continuous operator such that $\nabla_u\eta_0(0,\lambda)=0,\ \nabla^2_u\eta_0(0,\lambda)=0$ for every $\lambda\in\bR$.
\end{enumerate}
\end{Lemma}

Let us denote  by $\sigma(-\Delta; \cU) = \{ 0=\beta_1 < \beta_2 < \ldots < \beta_k < \ldots\}$ the set of distinct eigenvalues of the Laplace operator (with Neumann boundary conditions) on $\cU$. Write $\bV_{-\Delta}(\beta_k)$ for the eigenspace of $-\Delta$ corresponding to  $\beta_k \in \sigma(-\Delta; \cU)$.
Let us denote by $\bH_k$ the space $\bigoplus_{i=1}^p \bV_{-\Delta}(\beta_k).$ By the spectral properties of self-adjoint, completely continuous operators, it follows that $H^1(\cU) = cl ( \bigoplus_{k=1}^{\infty} \bV_{-\Delta} (\beta_k)).$ Let $\bH^0 = \{0\}$ and $\pi_n \colon \bH \rightarrow \bH$ be a natural $G$-equivariant projection such that $ \pi_n(\bH) = \bH^n$ for $n \in \bN\cup\{0\}$, where $\bH^n =  \bigoplus_{k=1}^n \bigoplus_{j=1}^{p} \bV_{-\Delta} (\beta_k)$. 
This defines a $G$-equivariant approximation scheme $\{\pi_n\colon\bH \rightarrow \bH\colon n \in \bN \cup \{0\} \}$ and therefore the assumptions on $\bH$ of Subsection \ref{subsec:glob} are satisfied.
In particular, for every $u\in \bH$ there exists a unique sequence $\{u^k\}$ such that $u^k\in\bH_k$ and $u=\sum_{k=1}^{\infty}u^k$.

Put $\bW=(T_{u_0} G (u_0))^{\bot}\subset\bR^p$, $H=G_{u_0}$. Let $p_0=\dim G(u_0)$, $p_1=\dim \bW^H$ and $p_2=\dim(\bW^H)^{\bot}$.
It is known (see \cite{Geba}) that
\begin{equation*}
\left.\begin{array}{rccc}
& T_{u_0} G(u_0)& &T_{u_0} G(u_0) \\
&\oplus& &\oplus\\
A\colon &\bW^H&\to&\bW^H\\
&\oplus& &\oplus\\
&(\bW^H)^{\bot}& &(\bW^H)^{\bot}
\end{array}\right.
\end{equation*}
has the following form
\begin{equation}\label{eq:matrixA}
A=\left[\begin{array}{ccc}
0&0&0\\
0&B(u_0)&0\\
0&0&C(u_0)
\end{array}\right],
\end{equation}
where $B(u_0)$ is of dimension $p_1\times p_1$  and $C(u_0)$ of $p_2\times p_2$.

Let us consider an orthonormal basis of $\bR^p$ consisting of eigenvectors of the matrix $A$. More precisely, denote by $f_1,\ldots, f_{p_0}$ the eigenvectors corresponding to 0, by $f_{p_0+1},\ldots,f_ {p_0+p_1}$ the eigenvectors corresponding to (not necessarily distinct) eigenvalues $b_{p_0+1},\ldots, b_{p_0+p_1}$ of $B(u_0)$ and by $f_{p_0+p_1+1},\ldots,f_ {p}$ the eigenvectors corresponding to eigenvalues $ c_{p_0+p_1+1},\ldots, c_{p}$ of $C(u_0)$.

Let $\tau_j\colon\bH\to H^1(\cU)$ be a projection such that $\tau_j(u)(x)=(u(x),f_j)$, $j=1,\ldots, p$.
Clearly, if $u^k \in \bH_k,$ then $\tau_j(u^k) \in \bV_{-\Delta}(\beta_k)$. Moreover, every $u\in\bH$ can be represented by $\displaystyle{u=\sum_{k=1}^{\infty}\sum_{j=1}^p \tau_j(u^k)\cdot f_j.}$
Denote
\begin{align*}
\cH_{0}&=\big\{u \in \bH \colon u = \sum\limits_{k=1}^{\infty}\sum\limits_{j=1}^{p_0} \tau_j(u^k) \cdot f_j\big\},\\
\cH_{1}&=\big\{u \in \bH \colon u = \sum\limits_{k=1}^{\infty}\sum\limits_{j=p_0+1}^{p_0+p_1} \tau_j(u^k) \cdot f_j\big\},\\ 
\cH_{2}&=\big\{u \in \bH \colon u = \sum\limits_{k=1}^{\infty}\sum\limits_{j=p_0+p_1+1}^{p} \tau_j(u^k) \cdot f_j\big\}.
\end{align*}
Note that for $u=(\mathbf{u}_0,\mathbf{u}_1,\mathbf{u}_2)\in\cH_{0}\oplus\cH_{1}\oplus\cH_{2}=\bH$ we have \begin{align*}L_{\lambda A}u=(L_{0}\mathbf{u}_0,L_{\lambda B(u_0)}\mathbf{u}_1,L_{\lambda C(u_0)}\mathbf{u}_2),
\end{align*}
where
\begin{align*}
\langle L_{0}\mathbf{u}_0,v_0\rangle_{\cH_{0}}&=\int_{\cU} (\mathbf{u}_0(x),\mathbf v_0(x))\, dx, \\
\langle L_{\lambda B(u_0)}\mathbf{u}_1,v_1\rangle_{\cH_{1}}&=\int_{\cU} (\mathbf{u}_1(x)+\lambda B(u_0)\mathbf{u}_1(x),\mathbf v_1(x))\, dx,\\
\langle L_{\lambda C(u_0)}\mathbf{u}_2,v_2\rangle_{\cH_{2}}&=\int_{\cU} (\mathbf{u}_2(x)+\lambda C(u_0)\mathbf{u}_2(x),\mathbf v_2(x))\, dx
\end{align*}
for all $\mathbf v_0\in \cH_{0}$, $\mathbf v_1\in \cH_{1}$, $\mathbf v_2\in \cH_{2}$.

In the lemma below we characterise the operators $L_{0}$, $L_{\lambda B(u_0)}$ and $L_{\lambda C(u_0)}$. The proof is analogous to the one of Lemma 3.2 of \cite{GolKlu}.

\begin{Lemma}\label{lem:operatorLlambdaA}
For every $u=(\mathbf{u}_0, \mathbf{u}_1, \mathbf{u}_2)\in\cH_{0}\oplus \cH_{1} \oplus \cH_{2}$
\begin{align*}
L_{0} \mathbf{u}_0=&\sum\limits_{k=1}^{\infty}\sum\limits_{j=1}^{p_0} \frac{1}{1+\beta_k} \tau_j(u^k) \cdot f_j,\\
L_{\lambda B(u_0)} \mathbf{u}_1=&\sum\limits_{k=1}^{\infty}\sum\limits_{j=p_0+1}^{p_0+p_1} \frac{1+\lambda b_j}{1+\beta_k} \tau_j(u^k) \cdot f_{j}, \\
L_{\lambda C(u_0)} \mathbf{u}_2=&\sum\limits_{k=1}^{\infty}\sum\limits_{j=p_0+p_1+1}^{p} \frac{1+\lambda c_j}{1+\beta_k} \tau_j(u^k) \cdot f_{j},
\end{align*}
where $u=\sum\limits_{k=1}^{\infty}u^k$.
\end{Lemma}

As an easy consequence of Lemma \ref{lem:operatorLlambdaA} we obtain:

\begin{Corollary}\label{cor:spektrumLA}
The spectrum of $Id-L_{\lambda A}$ is given by:
\[
\sigma(Id-L_{\lambda A})=\left\{
 \frac{\beta_k-\lambda\alpha_j}{1+\beta_k}\colon \alpha_j \in\sigma(A), \beta_k\in \sigma(-\Delta; \cU)
\right\}.
\]
\end{Corollary}

From Lemma \ref{lem:postacPhi} it follows that $\tilde{u}_0$ is a critical point of $\Phi$ for every $\lambda\in\bR$. From the $G$-invariance of $\Phi$ we obtain that $G(\tilde{u}_0)\times\bR\subset (\nabla\Phi)^{-1}(0)$. We call the elements of the family $G(\tilde{u}_0)\times\bR$ the trivial solutions of the system \eqref{eq:system}. We are going to study global bifurcations of nontrivial solutions from this family.

We start with the necessary condition. Let us denote by $\Lambda$ the set of all $\lambda \in\bR$ such that the orbit $G(\tilde{u}_0)\times\{\lambda\}$ is degenerate, i.e. such that $\dim \ker \nabla^2_u\Phi(\tilde{u}_0, \lambda) >\dim(G(\tilde{u}_0) \times \{\lambda\}).$

\begin{Lemma}\label{lem:nec}
If a bifurcation of solutions of \eqref{eq:system} occurs from the orbit $G(\tilde{u}_0) \times \{\lambda_0\}$, then $\lambda_0 \in \Lambda.$ Moreover,
$$\Lambda=\bigcup_{\alpha_j \in \sigma(A) \setminus\{0\}} \bigcup_{\beta_k \in \sigma(-\Delta;\cU)} \left\{\frac{\beta_k}{\alpha_j}\right\}.$$
\end{Lemma}
The proof of this lemma is similar to the proof of Lemma 3.1 of \cite{GolKluSte}.

Denote by $\mu_M(a)$ the multiplicity of the eigenvalue $a$ of a matrix $M$. Recall that $B(u_0)$ is given by $B(u_0)=A_{|\bW^H}$, see \eqref{eq:matrixA}.

\begin{Theorem}\label{thm:globellip}
Consider the system \eqref{eq:system} with the potential $F$ and $u_0$ satisfying the assumptions (a1)--(a5).
Fix $\lambda_{0} \in \Lambda$ and assume that $\sigma(\lambda_0 B(u_0))\cap \sigma(-\Delta; \cU)=\{b_{j_1},\ldots, b_{j_s}\}$. If
\begin{equation}\label{eq:zalthm}
\sum_{i=1}^s \mu_{B(u_0)}(b_{j_i}) \dim \bV_{-\Delta}(b_{j_i}) \text{ is odd,}
\end{equation}
then a global bifurcation of solutions of \eqref{eq:system} occurs from the orbit $G(\tilde{u}_0) \times \{\lambda_0\}$.
\end{Theorem}

\begin{proof}
To prove the existence of the global bifurcation in the case $\lambda_0 \neq 0$ we use Theorem \ref{thm:globbifV}. From the above considerations the assumptions  on the space $\bH$ and the functional $\Phi$ of this theorem are satisfied.
Moreover from Lemma \ref{lem:nec} it follows that we can choose $\varepsilon>0$ such that $\Lambda\cap[\lambda_0-\varepsilon,\lambda_0+\varepsilon]=\{\lambda_0\}$.

As in the previous section, denote by $\cW(\lambda)$ the direct sum of the eigenspaces of $\nabla^2_u\Phi(u_0,\lambda)=Id-L_{\lambda A}$ corresponding to the negative eigenvalues and note that, by Corollary \ref{cor:spektrumLA}, it can be represented by the formula
$$ \cW(\lambda)=
\bigoplus_{\alpha_j\in\sigma( A)} \ \bigoplus_{\substack{\beta_k \in \sigma(-\Delta;\cU)\\\beta_k<\lambda  \alpha_j}} \bV_{-\Delta}(\beta_k)^{\mu_{ A}(\alpha_{j})}.$$

In the above formula we understand $\bV_{-\Delta}(\beta_k)^{\mu_{A}(\alpha)}$ as $\mathrm{span}\{
h\cdot f\colon h\in \bV_{-\Delta}(\beta_{k}), f\in \bV_{A}(\alpha)\}$, see also \cite{GolKluSte}. However, this space is isomorphic with the direct sum of $\mu_A(\alpha)$ copies of $\bV_{-\Delta}(\beta_k)$.
Since in our computations the important thing is the dimension of the spaces, we identify $\bV_{-\Delta}(\beta_k)^{\mu_{A}(\alpha)}= \bigoplus_{i=1}^{\mu_{A}(\alpha)}\bV_{-\Delta}(\beta_k)$.

For $\lambda\in\bR$ denote by $\cV(\lambda)$ the $G$-representation being the eigenspace  of $Id-L_{\lambda A}$ corresponding to the zero eigenvalue. Note that, from Lemma \ref{lem:operatorLlambdaA} and Corollary \ref{cor:spektrumLA}, we have
$$\cV(\lambda)=\bigoplus_{\alpha_j\in\sigma(A)} \ \bigoplus_{\substack{\beta_k \in \sigma(-\Delta;\cU)\\ \beta_k=\lambda  \alpha_j}} \bV_{-\Delta}(\beta_k)^{\mu_{ A}(\alpha_{j})}.$$
Additionally by \eqref{eq:matrixA} and by the definition of the action of $G$ on $\bH$,
\begin{equation}\label{eq:vlambdaH}
\cV(\lambda)^H=\bigoplus_{b_j\in\sigma(B(u_0))} \ \bigoplus_{\substack{\beta_k \in \sigma(-\Delta;\cU)\\ \beta_k=\lambda b_j}} \bV_{-\Delta}(\beta_k)^{\mu_{B(u_0)}(b_{j})}.
\end{equation}

Note that if $\lambda_0>0$ then $\cW(\lambda_0+\varepsilon)=\cW(\lambda_0) \oplus \cV(\lambda_0)$, $\cW(\lambda_0-\varepsilon)=\cW(\lambda_0)$. 
By the assumption \eqref{eq:zalthm}, from \eqref{eq:vlambdaH} it follows that $\dim \cV(\lambda_0)^H$ is odd.
Applying Theorem \ref{thm:globbifV}, we obtain the assertion.
In the case $\lambda_0\leq 0$ the proof is analogous.
To prove the assertion for $\lambda_0=0$ we use the fact that $\bV_{-\Delta}(0)$ is a one dimensional trivial $G$-representation and we apply Theorem \ref{thm:RabAlt}.
\end{proof}

\section{Examples}
In this section we discuss a few examples in order to illustrate the abstract results proved in the previous section. 

\nt \it{Example 1.} \rm We start with the system \eqref{eq:system} on a rectangle $\cU=[0,l_1]\times[0,l_2]$. The set $\cU$ has a piecewise $C^1$ boundary, hence the unit outward normal vector $\nu(x)$ is well defined for almost all $x\in\partial\cU$. 
It is known that the eigenvalues of the Neumann Laplacian are of the form
\begin{equation*}
\beta_{km}=\pi^2\left(\frac{k^2}{l_1^2}+\frac{m^2}{l_2^2}\right)
\end{equation*}
with the corresponding eigenfunction
\begin{equation*}
v_{km}(x,y)=\cos\left(\frac{k \pi}{l_1}x\right)\cos\left(\frac{m \pi}{l_2}y\right)
\end{equation*}
for $k,m\in \bN\cup\{0\}$, see chapter IV\S2 of \cite{Michlin}.

Note that if $\beta_{k_1m_1}=\beta_{k_2 m_2}$, then 
\[
\frac{k_1^2}{l_1^2}+\frac{m_1^2}{l_2^2}=\frac{k_2^2}{l_1^2}+\frac{m_2^2}{l_2^2}
\]
and consequently
\[
l_1^2(m_1^2-m_2^2)+l_2^2(k_1^2-k_2^2)=0.
\]
If we assume additionally that $l_1^2\neq ql_2^2$ for every $q\in\bQ$, then $k_1=k_2$ and $m_1=m_2$.

We have thus proved the following lemma:

\begin{Lemma}
If $l_1^2\neq ql_2^2$ for every $q\in\bQ$, then $\dim \bV_{-\Delta}(\beta)=1$ for every $\beta\in\sigma(-\Delta, \cU)$.
\end{Lemma}

Therefore, knowing that the dimensions of all the eigenspaces are equal one, we verify the assumption \eqref{eq:zalthm} of Theorem \ref{thm:globellip} by counting the multiplicities of eigenvalues of the matrix $B(u_0)$. More precisely, we have the following:

\begin{Theorem}
Consider the system \eqref{eq:system} on $\cU=[0,l_1]\times[0,l_2]$ with the potential $F$ and $u_0$ satisfying the assumptions (a1)--(a5).
Assume that $l_1^2\neq ql_2^2$ for every $q\in\bQ$.
Fix $\lambda_{0} \in \Lambda$ and assume that $\sigma(\lambda_0 B(u_0))\cap \sigma(-\Delta; \cU)=\{b_{j_1},\ldots, b_{j_s}\}$.
If
$\sum_{i=1}^s \mu_{B(u_0)}(b_{j_i})$ is odd,
then a global bifurcation of solutions of \eqref{eq:system} occurs from the orbit $G(\tilde{u}_0) \times \{\lambda_0\}$.
\end{Theorem}

\begin{Remark}
In the above theorem we have used the property that the eigenspaces of the Laplacian on a given domain are one-dimensional. However, this property holds not only in this case. It is known, that this property is generic, see for example \cite{Albert1975}, \cite{Albert1978} and \cite{Uhlenbeck}. Therefore, considering an arbitrary domain $\cU$, the assumption \eqref{eq:zalthm} of Theorem \ref{thm:globellip} can be usually simplified as in the above theorem, i.e. to counting the multiplicities of eigenvalues of the matrix $B(u_0)$.

On the other hand, there are known domains $\cU$ such that the eigenspaces of the Laplacian are of greater dimensions. This occurs for example in the case of symmetric domains.
Such eigenspaces are described for instance in  \cite{GolKluSte} (for a ball), \cite{Shimakura} (for a sphere), and in \cite{Bang}, \cite{RybShiSte} (for a geodesic ball). The case of a more general symmetric space can be found in \cite{Gurarie}. Note that in some cases the explicit formulae for the dimensions of eigenspaces are given. This allows to verify the assumption \eqref{eq:zalthm} of Theorem \ref{thm:globellip}.
\end{Remark}

\nt \it{Example 2.} \rm Consider now the system \eqref{eq:system} with an arbitrary domain $\cU$ and assume that $G=SO(p)$, i.e. $\bR^p$ is an orthogonal $SO(p)$-representation. Suppose that $u_0=(0,\ldots,0,\hat{u}_0)$ for some $\hat{u}_0\in\bR\setminus\{0\}$ and the assumptions (a1)--(a5) are satisfied. It is easy to see that $H=SO(p)_{u_0}=SO(p-1)\times\{1\}$ and $G/H$ is $G$-homeomorphic with a $(p-1)$-dimensional sphere. Moreover, $T_{u_0} G (u_0)$ is a $(p-1)$-dimensional plane, $\bW=(T_{u_0} G (u_0))^{\bot}=\{0\}^{p-1}\times\bR$ and $\bW^H=\bW$.

Hence \begin{equation*}
\left.\begin{array}{rccc}
& T_{u_0} G(u_0)& &T_{u_0} G(u_0) \\
A\colon &\oplus& \to &\oplus\\
&\bW^H& &\bW^H\
\end{array}\right.
\end{equation*}
has the following form
\begin{equation*}
A=\left[\begin{array}{cc}
0&0\\
0&b_0
\end{array}\right].
\end{equation*}
In particular, $\sigma(\lambda B(u_0))=\{\lambda b_0\}$ and $\Lambda= \bigcup_{\beta_k \in \sigma(-\Delta;\cU)} \left\{\frac{\beta_k}{b_0}\right\}$.

\begin{Theorem}\label{thm:exsop}
Consider the system \eqref{eq:system} with the potential $F$ and $u_0$ satisfying the assumptions (a1)--(a5).
Assume that $u_0=(0,\ldots,0,\hat{u}_0)$ for some $\hat{u}_0\in\bR\setminus\{0\}$.
Fix $\lambda_{0}=\frac{\beta_k}{b_0} \in \Lambda$. If
$ \dim \bV_{-\Delta}(\lambda_0 b_{0})$ is odd,
then a global bifurcation of solutions of \eqref{eq:system} occurs from the orbit $G(\tilde{u}_0) \times \{\lambda_0\}$.
\end{Theorem}

\begin{Remark}
It is easy to see that if we replace $u_0=(0,\ldots,0,\hat{u}_0)$ by an arbitrary $v_0\in\bR^p\setminus\{0\}$, then the assertion of the above theorem holds.
\end{Remark}

\end{document}